\documentclass[twoside,a4paper,reqno,11pt]{amsart} 
\usepackage[top=28mm,right=30mm,bottom=28mm,left=30mm]{geometry}

\usepackage{amsfonts, amsmath, amssymb, mathrsfs, bm, latexsym, stmaryrd, array, mathtools, bold-extra,xcolor}

\usepackage{etaremune}
\usepackage{enumerate}
\usepackage[pagebackref]{hyperref}

\newcommand{\la}{\langle}
\newcommand{\ra}{\rangle}

\newcommand{\leqs}{\leqslant}
\newcommand{\geqs}{\geqslant}

\newcommand{\Aut}{\operatorname{Aut}}

\newcommand{\PSU}{\mathrm {U}}

\newcommand{\PSL}{\mathrm {L}}
\newcommand{\PSp}{\mathrm {S}}
\newcommand{\PSO}{\mathrm {O}}

\newcommand{\OB}{\mathbf{O}}
\newcommand{\CB}{\mathbf{C}}
\newcommand{\FB}{\mathbf{F}}
\newcommand{\NB}{\mathbf{N}}
\newcommand{\ZB}{\mathbf{Z}}
\newcommand{\RB}{\mathbf{R}}
\newcommand{\EB}{\mathbf{E}}

\newcommand{\GL}{\mathrm {GL}}

\newcommand{\Irr}{\mathrm {Irr}}
\newcommand{\Alt}{\mathrm {Alt}}
\newcommand{\Sym}{\mathrm {Sym}}
\newcommand{\SSS}{\mathrm {S}}
\newcommand{\AAA}{\mathrm {A}}

\newcommand{\supp}{\mathrm {supp}}
\makeatletter
\newcommand{\imod}[1]{\allowbreak\mkern4mu({\operator@font mod}\,\,#1)}
\makeatother

\renewcommand{\leq}{\leqs}
\renewcommand{\geq}{\geqs}

\newtheorem{thm}{Theorem}[section] 
 
\newtheorem{lem}[thm]{Lemma}

\theoremstyle{definition}
\newtheorem{rem}[thm]{Remark}

\begin{document}
\title[Orders of commutators]{Orders of commutators and products of conjugacy classes in finite groups}

\author{Hung P. Tong-Viet}
\address{H.P. Tong-Viet, Department of Mathematics and Statistics, Binghamton University, Binghamton, NY 13902-6000, USA}
\email{htongvie@binghamton.edu}

\begin{abstract}
Let $G$ be a finite group, let $x \in G$, and let $p$ be a prime. We prove that the commutator $[x,g]$ is a $p$-element for every $g \in G$ if and only if $x$ is central modulo $\mathbf{O}_p(G)$, where $\mathbf{O}_p(G)$ denotes the largest normal $p$-subgroup of $G$. This result provides a common generalization of certain variants of both the Baer--Suzuki theorem and Glauberman's $\mathbf{Z}_p^*$-theorem. As an application, we show that if $K$ is a conjugacy class of $G$ such that $K^{-1}K = 1 \cup D \cup D^{-1}$ for some conjugacy class $D$ of $G$, then the subgroup generated by $K$ is solvable.
\end{abstract}

\date{\today}

\maketitle

\section{Introduction}\label{s:intro}

Let $G$ be a finite group and let $x$ be an element of $G$. A central question in group theory is whether $x$ lies in a small characteristic subgroup of $G$. This problem has been the subject of extensive research, and numerous papers have addressed various aspects of it. Below, we highlight two of the best-known results on this topic.
Recall that, for a prime $p$, an element $x$ is called \emph{$p$-singular} if its order is divisible by $p$; it is called \emph{$p$-regular} (or a $p'$-element) if its order is not divisible by $p$; and it is called a \emph{$p$-element} if its order is a power of $p.$ 

The Baer--Suzuki theorem provides a useful characterization of the Fitting subgroup $\FB(G)$ of a finite group $G$, which is the largest nilpotent normal subgroup of $G$. It states that an element $x\in G$ lies in $\FB(G)$ if and only if $\la x,x^g\ra$ is nilpotent for every $g\in G.$ Various versions of this result  have been established by Baer, Suzuki, Alperin and Lyons (see \cite{Baer,Suzuki,AL}).
When $x$ is a $p$-element for some prime $p$, the theorem can be restated as follows: $x\in \OB_p(G)$ if and only if $\la x,x^g\ra$ is a $p$-group for every $g\in G$, where $\OB_p(G)$ is the largest normal $p$-subgroup of $G$.  This result was further strengthened by Guralnick and Robinson \cite[Theorem A]{GR}, who showed that if $x$ is an element of prime order $p$ and the commutator $[x,g]=x^{-1}g^{-1}xg$ is a $p$-element for every $g\in G$, then $x\in \OB_p(G)$. Later, Guralnick and Malle extended this result to  all $p$-elements  \cite[Theorem 1.4]{GM1}.  The various proofs of the Baer--Suzuki theorem mentioned above do not rely on the classification of finite simple groups, whereas the variations established in \cite{GR,GM1} do depend on it. For further generalizations of the Baer--Suzuki theorem, see \cite{GTT,Wielandt}.

In the opposite direction, Glauberman's $\ZB_p^*$-theorem provides a criterion for determining when a $p$-element of a finite group lies in the center of the group modulo a certain normal subgroup. Recall that an element $x\in G$ is said to be central modulo a normal subgroup $N$ of $G$ if the image $xN$ lies in the center of the quotient group $G/N$. Equivalently, the subgroup $[x,G]$, generated by all the commutators $[x,g] $ for $g\in G$, is contained in $N$. For a prime  $p$, we define $\ZB_p^*(G)$ to be the full inverse image in $G$ of the center of $G/\OB_{p'}(G)$, where $\OB_{p'}(G)$ is the largest normal $p'$-subgroup of $G$. Glauberman's $\ZB_p^*$-theorem states that if $x$ is a $p$-element of a finite group $G$, then $[x,g]$ is a $p'$-element for every $g\in G$ if and only if $x$ is central modulo $\OB_{p'}(G)$.  Equivalently, it can be rephrased as follows:  $x^G\cap \CB_G(x)=\{x\}$ if and only if $x\in \ZB_p^*(G)$, where $x^G$ denotes the conjugacy class of $G$ containing $x$ and $\CB_G(x)$ is the centralizer of $x$ in $G$. (See \cite[Theorem 5.1]{GTT}.)

Glauberman \cite{Glauberman} originally proved his theorem for $p=2$  using modular representation theory, so his proof does not depend on the classification of finite simple groups. For odd primes, the result was independently established by Guralnick and Robinson  \cite{GR} and by Artemovich \cite{Ar}; both proofs use the classification of finite simple groups.  We refer the reader to \cite{GTT} for various equivalent formulations  of this theorem. As an application, Guralnick and Robinson proved the following generalization (\cite[Theorem D(ii)]{GR}): if $x\in G$ is an element of prime order $r$ and $p\neq r$ is a prime such that $[x,g]$ is a $p$-element for every $g\in G,$ then $x$ is central modulo $\OB_p(G)$.
We now present our first result, which generalizes several of the  theorems discussed above.

\begin{thm}\label{th:r-elements-commutators}
Let $G$ be a finite group, let $x\in G$ and let $p$ be a prime. Then $[x,g]$ is a $p$-element for every $g\in G$ if and only if $x$ is central modulo $\OB_p(G)$.
\end{thm}

Theorem \ref{th:r-elements-commutators} generalizes  both Theorem D(ii) in \cite{GR} and Theorem 1.4 in \cite{GM1}, as discussed earlier. First, if we take $x\in G$ to be an element of prime order $r\neq p$, then  Theorem  \ref{th:r-elements-commutators} directly yields Theorem D(ii) of \cite{GR}. Next, suppose that $x\in G$ is a $p$-element and that $[x,g]$ is a $p$-element for every $g\in G$. Then, by Theorem \ref{th:r-elements-commutators},  $\la x\ra \OB_p(G)$ is a normal $p$-subgroup of $G$. Since $\OB_p(G)$ is the largest normal $p$-subgroup of $G$, it follows that $\la x\ra \OB_p(G)=\OB_p(G)$ and hence  $x\in \OB_p(G)$. This recovers Theorem 1.4 of \cite{GM1}.

To prove Theorem \ref{th:r-elements-commutators}, we first reduce to the case in which the group is almost simple. To address this case, we require the following result, which is of independent interest. Recall that a finite group $G$ is called almost simple with socle $S$ if $S\unlhd G\leq \Aut(S)$ for some finite nonabelian simple group $S$.

\begin{thm}\label{th:r-singular}
Let $G$ be a finite almost simple group, let $x$ be a nontrivial element of $G$ and let $p$ be a prime. Then  there exists $g\in G$ such that $[x,g]$ is a nontrivial $p'$-element.
\end{thm}

 For the proof of Theorem \ref{th:r-singular}, a direct calculation eliminates the cases in which the socle of $G$ is a large alternating group. In the remaining cases, we use character theory, in particular the structure constant formula, to determine whether an element can be written as a commutator of the form $[x,g]$.

\begin{rem}\label{rem1} We record here some remarks on Theorem \ref{th:r-elements-commutators}.
\begin{enumerate}
\item[(i)] If  $x\in G$ and $[x,g]$ is either $1$ or $p$-singular for every $g\in G,$ then $x$ need not be central modulo $\OB_p(G)$. For example, consider $G=\GL_2(3)$ and $x\in G$ of order $8$. Then the commutator $[x,g]$ has order $1$, $4$ or $6$ for every $g\in G$, but $G/\OB_2(G)\cong\SSS_3$ and so $x$ is not central modulo $\OB_2(G)$.  This example is taken from \cite{GTT}.

\item[(ii)]   It is conjectured in \cite{GR,GTT} that if $x\in G$ is an element of prime order $p$  and $[x,g]$ is either $1$ or $p$-singular for every $g\in G$, then $x\in \OB_p(G)$. This conjecture is false if $x$ is not of prime order by the  example above. The conjecture holds if $\OB_p(G)$ is abelian (see \cite[Theorem 6.2]{GTT}).  

\end{enumerate}
\end{rem}

We next present some applications of the preceding theorems to problems concerning products of conjugacy classes. 
 Let $G$ be a finite group and let $X\subseteq G$ be a nonempty set. For any $g\in G$, we define  
 $X^g=\{x^g:=g^{-1}xg:x\in X\}$, and   $X^{-1}=\{x^{-1}:x\in X\}$. We say that $X$ is a normal subset of $G$ if $X^g=X$ for all $g\in G.$ Clearly, every conjugacy class of $G$ is a normal subset of $G$, and conversely,  every normal subset of $G$ is a disjoint union of conjugacy classes of $G$. Given two nonempty subsets $X,Y\subseteq G$, we define their (set) product as $XY=\{xy:x\in X,y\in Y\}$. If $X$ and $Y$ are normal subsets of $G$, then  $X^{-1}$ and $XY$ are normal subsets of $G$ and we have the equality $XY=YX$. Let $K=x^G$ be the conjugacy class of $G$ containing an element $x\in G$.  Theorem \ref{th:r-elements-commutators} can be equivalently reformulated  as follows: `$K^{-1}K$ consists entirely of $p$-elements if and only if $x$ is central modulo $\OB_p(G)$.' We now investigate the structure of finite groups under natural restrictions on how the product of two conjugacy classes decomposes.

The study of products of conjugacy classes in finite groups has become an active area of research in finite group theory, garnering significant attention in recent years. For an overview of this topic, we refer the reader to the survey by Arad and Herzog \cite{AH} and the more recent exposition in \cite{BFM2}. In 1985, Arad and Herzog \cite{AH} proposed the following conjecture: in a finite nonabelian simple group, the product of any two nontrivial conjugacy classes is never a single conjugacy class.
Despite considerable interest, this conjecture remains open, particularly in the case of finite simple classical groups. Various partial results have been established in support of the conjecture (see \cite{AH,FA,MTong,GMT,GN}), but a general proof remains elusive.
In a related direction, Thompson conjectured that if $G$ is a finite nonabelian simple group, then there exists a conjugacy class $C\subseteq G$ such that $C^2=G$. This conjecture has recently been resolved (see \cite{LT1, LT2, LWZ}).

An element $g\in G$ is real if it is conjugate to its inverse, that is, if there exists an element $t\in G$ such that $g^t=g^{-1}$. A conjugacy class $C$ of $G$ is called a real class if it contains a real element; in this case, we have $C=C^{-1}.$ 
In \cite{BFM3}, the authors proposed the following conjecture:  if a finite group $G$ possesses a conjugacy class $K$  such that $K^{-1}K=1\cup D\cup D^{-1}$ for some conjugacy class $D$ of $G$, then the subgroup $\la K\ra$ generated by $K$ is solvable. In the same paper, it is shown that  $G$ cannot be a simple group. Furthermore, under the assumption that $D=D^{-1}$, the authors prove that  $\la K\ra$ is solvable provided certain conditions on the sizes of $D$ and $K$ are met. When $K$ is real,  $D$ is real and hence $K^2=1\cup D$. In this case, the solvability of $\la K\ra$ follows directly from \cite[Theorem A]{BFM4}. In the present work, we confirm that this conjecture holds in full generality.

\begin{thm}\label{th:3-classes}
Let $G$ be a finite group and let $K$ be a conjugacy class of $G$. Suppose $K^{-1}K=1\cup D\cup D^{-1}$ for some conjugacy class $D$ of $G$. Then $\la K\ra$ is solvable.
\end{thm}

Finally, we prove the following result, whose proof does not use the classification of finite simple groups. It can be viewed as a special case of Conjecture 2 in \cite{GTT}.

\begin{thm}\label{th:two-primes-divisors}
 Let $G$ be a finite group and let $x\in G$ be a $p$-element for some prime $p$. Assume that there exist two distinct primes $r$ and $s$ such that, for every $g\in G$, either $[x,g]=1$ or the order of $[x,g]$ is divisible by $rs$. Then $x\in \ZB(G)$.
\end{thm}

The paper is  structured as  follows. In Section \ref{sec2}, we present the necessary preliminaries. The proofs of Theorems \ref{th:r-elements-commutators} and \ref{th:r-singular} are given in Section \ref{sec3}. Finally, Theorems \ref{th:3-classes} and \ref{th:two-primes-divisors} are proved in Section \ref{sec4}. 

\medskip

Our notation is standard. We follow \cite{Isaacs-fg} for  finite group theory, \cite{Isaacs-ct} for the character theory of finite groups, and \cite{ATLAS} for the notation of finite simple groups. Let $p$ be a prime and let $n\ge 1$ be an integer. The $p$-part of $n$, denoted $n_p$, is the largest power of $p$ dividing $n,$ so that $n=n_pn_{p'}$, where $n_{p'}$ is relatively prime to $p$. We denote by $\pi(n)$ the set of all distinct prime divisors of $n.$ If $G$ is a finite group, we write $\pi(G):=\pi(|G|)$ for the set of all distinct prime divisors of the order of $G$. 

\section{Preliminaries}\label{sec2}

Let  $G$ be a finite group. A component of $G$ is a finite quasisimple subnormal subgroup of $G$. Recall that  a finite group $L$ is quasisimple if it is perfect, that is, $L=L'$, and  $L/\ZB(L)$ is a finite nonabelian simple group. The layer of $G$, denoted by $\EB(G)$, is the normal subgroup of $G$ generated by all components of $G$.  The generalized Fitting subgroup of $G$, denoted by $\FB^*(G)$, is defined as the product $\FB^*(G)=\FB(G)\EB(G)$. As usual, $\ZB(G)$ denotes the center of $G$.  For a subgroup $H$ of $G$, the centralizer and normalizer of $H$ in $G$ are denoted by $\CB_G(H)$ and $\NB_G(H)$, respectively. We denote by $\RB(G)$ the solvable radical of $G$, that is, the largest normal solvable subgroup of $G$.

Burnside's $p^\alpha$-lemma (\cite[Theorem 3.9]{Isaacs-ct}) states that if $G$ is a finite nonabelian simple group, then the trivial class is the only conjugacy class of $G$ of prime power size. In particular, if a finite group $G$ has a nontrivial conjugacy class of prime power size, then $G$ is not simple. Kazarin extended this as follows.

\begin{lem}\label{lem:pa}\emph{(\cite[Theorem]{Kazarin}).}
Let $G$ be a finite group and let $x\in G.$ If $|x^G|$ is a power of a prime, then $\la x^G\ra$ is a solvable group.
\end{lem}

We need the following result due to Shult. 
\begin{lem}\label{lem:central}\emph{(\cite[Theorem~1]{Shult}).}
Let $A$ be a solvable subgroup of a finite group $G$ and suppose that $A$ lies in the center of every solvable subgroup of $G$ containing $A$. Then $A$ lies in the center of $G$.
\end{lem}

We note that the proofs of Lemmas \ref{lem:pa} and \ref{lem:central} do not rely on the classification of finite simple groups.

For a finite group $G$, as usual, we denote by $\Irr(G)$ the set of complex irreducible characters of $G$. A character $\chi\in\Irr(G)$ is said to have $p$-defect zero for a prime $p$ if the quotient $|G|/\chi(1)$ is not divisible by $p$; equivalently, $\chi(1)_p=|G|_p$. 
It is well known that a finite simple group $G$ of Lie type possesses an irreducible character of $p$-defect zero for each prime divisor $p$ of $|G|$ (see \cite{Michler,Willems}).  Inspecting  \cite{Michler,Willems} and the proofs of Propositions 2.4 and 2.5 in \cite{Malle}, we obtain the following.

\begin{lem}\label{lem:semisimple}
Let $S$ be a finite nonabelian simple group of Lie type in defining characteristic $r$. For each prime $p\neq r$ dividing $|S|$, $S$ has an irreducible character of $p$-defect zero whose degree is not divisible by $r$.
\end{lem}

 \begin{proof}  Let $p\neq r$ be a prime dividing $|S|$.
   
 Assume $S=\PSL_2(q)$ with $q\ge 4$. If $S\cong \PSL_2(4)\cong\PSL_2(5)$, then the result can be checked using GAP \cite{GAP}. Thus we may assume that $q\ge 7$. It is well known that $S$ has irreducible characters $\chi_1$ and $\chi_2$ of degree $q-1$ and $q+1$, respectively (see \cite[Theorems 38.1, 38.2]{Dornhoff}). Since $|S|=q(q^2-1)/\gcd(2,q-1)$ and $p\mid q^2-1$, either $\chi_1$ or $\chi_2$ has $p$-defect zero. Clearly, $r$ does not divide $\chi_i(1)$ for $i=1,2.$
 
 Assume $S=\PSL_3(q)$ with $q\ge 2.$ We have $$|S|=q^3(q-1)^2(q+1)(q^2+q+1)/\gcd(3,q-1).$$ Note that $\gcd(q-1,q^2+q+1)=\gcd(3,q-1)$. Using GAP \cite{GAP}, we may assume that $q\ge 8$.  By \cite[Table 2]{SF}, there exist $\chi_1,\chi_2\in\Irr(S)$ with $\chi_1(1)=(q-1)^2(q+1)$ and $\chi_2(1)=q^2+q+1$. Now it is easy to see that  for every prime $p\neq r$ dividing $|S|$, either $\chi_1$ or $\chi_2$ has $p$-defect zero and that $r$ does not divide $\chi_i(1)$ for $i=1,2$. 
 
 Assume $S=\PSU_3(q)$ with $q\ge 3.$ Again, we can assume  $q \ge 8$. We have $$|S|=q^3(q+1)^2(q-1)(q^2-q+1)/\gcd(3,q+1).$$ We can use the same argument as in the previous case using \cite{SF}. We see that $S$ has two irreducible characters $\chi_1$ and $\chi_2$ of degree $(q+1)^2(q-1)$ and $q^2-q+1$, respectively. For each prime $p\neq r$ dividing $|S|$, one of these characters has $p$-defect zero.
 
 Assume $S=\PSp_4(q)$ with $q\ge 3.$ We have $$|S|=q^4(q-1)^2(q+1)^2(q^2+1)/\gcd(2,q-1)$$ and $\gcd(2,q^2+1)=\gcd(2,q-1).$
 The character tables of $S$ are given in \cite{Enomoto} for $q$ even and in \cite{Srinivasan} for $q$ odd.  Using GAP \cite{GAP}, we may assume $q\ge 9$. Now $S$ has irreducible characters $\chi_1$ and $\chi_2$ of degree $(q^2-1)^2$ and $q^4-1$. It is readily seen that one of these two characters has $p$-defect zero, as required.

  Let $\mathcal{L}$ be the set consisting of the following simple groups: $$\PSp_4(2)',{}^2{\rm F}_4(2)',{}^2{\rm G}_2(3)'\cong \PSL_2(8),{\rm G}_2(2)'\cong \PSU_3(3),$$ $$\PSL_6(2),\PSL_7(2),\PSU_4(2),\PSU_6(2),\PSU_7(2),\PSp_6(2),\PSp_8(2),$$ and \[\PSO_8^\pm(2),\PSO_8^\pm(3),\PSL_2(q) (q\ge 4),\PSL_3(q) (q\geq 2),\PSU_3(q)(q\geq 3), \PSp_4(q) (q\ge 3)\] where $q$ is a power of $r.$

 If $S$ is one of the remaining groups in $\mathcal{L}$, then one can check using GAP \cite{GAP} that the lemma holds.
 
 Assume that $S\not\in \mathcal{L}$. We can find a simple algebraic group  $\mathbf{G}$ of simply connected type in defining characteristic $r$ and a Steinberg endomorphism $F: \mathbf{G}\rightarrow \mathbf{G}$   with finite group of fixed points   $G:=\mathbf{G}^F$ such that $S=G/\ZB(G)$. Let $\mathbf{G}^*$ be the dual group of $\mathbf{G}$ with corresponding Steinberg endomorphism  $F^*: \mathbf{G}^*\rightarrow \mathbf{G}^*$. Put $G^*:=\mathbf{G}^{*F^*}$, so that $|G^*|=|G|$. Let $\mathbf{T}\leq \mathbf{G}^*$ be an $F^*$-stable maximal torus and let $s\in T:=\mathbf{T}^{F^*}$ be a semisimple element with $\CB_{\mathbf{G}^*}(s)=\mathbf{T}$. Then the corresponding Deligne-Lusztig character $R_{T,s}$ is, up to sign, a semisimple  irreducible complex character of $G$ of degree $|G^*:T|_{r'}=|G|_{r'}/|T|$, which is relatively prime to $r$ (\cite[Corollary 7.3.5]{Carter}). If $s\in [G^*,G^*]$, then $R_{T,s}$ contains $\ZB(G)$ in its kernel, so it defines a character of $S$ (see \cite{Lusztig} and \cite[p. 174]{Malle}).

 Assume first that $S$ is an exceptional simple group of Lie type in characteristic $r$.
 It follows from the proof of Proposition 2.4 in \cite{Malle} that $G^*$ has two maximal tori $T_1$ and $T_2$ such that $$d=\gcd(|T_1|,|T_2|)=|\ZB(G)|, \text{and } \gcd(|T_1|/d,|T_2|/d)=1$$ and that each $T_i$ contains an element $s_i\in [G^*,G^*]$ of prime order with $\CB_{\mathbf{G}^*}(s_i)=\mathbf{T}_i$, where $T_i=\mathbf{T}_i^{F^*}$. Thus the corresponding Deligne-Lusztig characters $R_i:=R_{T_i,s_i}$ are irreducible characters of $S$, up to sign. Furthermore, for $i=1,2,$ $$R_i(1)=\frac{|G|_{r'}}{|T_i|}=\frac{|S|_{r'}}{|T_i|/d}.$$ Hence, for each prime $p\neq r$ dividing $ |S|$, either $R_1$ or $R_2$ has $p$-defect zero.

 Finally, assume that $S$ is a simple classical group of Lie type in characteristic $r$. Since $S$ is not in the set $\mathcal{L}$, it follows from the proof of Proposition 2.5 in \cite{Malle} that $G^*$ possesses two maximal tori $T_1$ and $T_2$ with the same properties as in the exceptional case. Therefore, the conclusion of the lemma follows as before.
  \end{proof}

 We also need the following result about the centralizers of Sylow $r$-subgroups of almost simple groups whose socle is a finite simple group of Lie type in characteristic $r$.
 
 \begin{lem}\label{lem:Sylow-centralizers}
 Let $G$ be a finite almost simple group with socle $S$, where $S$ is  a finite simple group of Lie type in characteristic $r$.  Let $R$ be a Sylow $r$-subgroup of $G$. Then $\CB_G(R)\leq R.$ 
 \end{lem}
 
 \begin{proof} Let $R_0=R\cap S$ and $H=\NB_G(R_0)$. Then $R_0$ is a Sylow $r$-subgroup of $S$ and $R_0\unlhd R$ as $S\unlhd G$. It follows that $R\leq H$ and thus $R_0\leq \OB_r(H)\leq R\leq  H$. Now $\CB_G(R)\leq \CB_G(R_0)\leq H$ and thus $\CB_G(R)=\CB_H(R).$ By \cite[Corollary 3.1.4]{GLS3}, $\FB^*(H)=\OB_r(H)$. By  \cite[31.13]{Aschbacher}, $\CB_H(\OB_r(H))\leq \OB_r(H)$ and hence $$\CB_G(R)=\CB_H(R)\leq \CB_H(\OB_r(H))\leq \OB_r(H)\leq R.$$
 The proof is now complete.
 \end{proof}

 The following lemma isolates an argument on components that is used in the proofs of Theorems \ref{th:r-elements-commutators} and \ref{th:3-classes}.

\begin{lem}\label{lem:component-moved}
Let $G$ be a finite group, let $x\in G$ and let $L$ be a component of $G$ such that $L^x\neq L$. Then for every prime $p$ dividing $|L/\ZB(L)|$, there exists a $p$-element $a\in L$ such that $[x,a]$ is a nontrivial $p$-element.
\end{lem}

\begin{proof}
Since $L$ and $L^x$ are distinct components of $G$, we have $[L,L^x]=1$ by \cite[31.5]{Aschbacher}. Let $p$ be a prime dividing $|L/\ZB(L)|$ and let $b\in L$ be a $p$-element. Then $b^x\in L^x$ commutes with $b^{-1}\in L$, and hence $[x,b^{-1}]=x^{-1}bxb^{-1}=b^xb^{-1}$ is a $p$-element. Thus $[x,a]$ is a $p$-element for every $p$-element $a\in L$.

Suppose that $[x,a]=1$ for every $p$-element $a\in L$. Then $x$ centralizes the subgroup $P$ of $L$ generated by all $p$-elements of $L$. Now $P\unlhd L$, and $P\not\leq \ZB(L)$ since $p$ divides $|L/\ZB(L)|$. As $L/\ZB(L)$ is simple, $L=P\ZB(L)$, and since $L$ is perfect, $L=[L,L]=[P,P]\leq P$. Hence $x$ centralizes $L$, so $L^x=L$, a contradiction.
\end{proof}
\section{Orders of commutators}\label{sec3}

Let $G$ be a finite group. If $C$ is a conjugacy class of $G$, we denote  its class sum by $\widehat{C}=\sum_{y\in C}y$. The set of all class sums $\widehat{C}$, where $C$ runs over the set of all conjugacy classes of $G$, forms a basis for the center $\ZB(\mathbb{C} G)$ of the complex group algebra $\mathbb{C} G$. For each $\chi\in\Irr(G)$, there exists a  $\mathbb{C}$-algebra homomorphism $\omega_\chi:\ZB(\mathbb{C} G)\rightarrow \mathbb{C}$ defined  by $$\omega_\chi(\widehat{C})=\frac{|C|\chi(g)}{\chi(1)}$$ where $g\in C$ (see \cite[Chapter 3]{Isaacs-ct}).   By \cite[Theorem 3.7]{Isaacs-ct}, $\omega_\chi(\widehat{C})$ is an algebraic integer.

For a nonempty set $\Omega$ of size $n$, we denote the alternating group and the symmetric group of degree $n$ acting on the set $\Omega$ by $\AAA_n=\Alt(\Omega)$ and $\SSS_n=\Sym(\Omega)$, respectively.
We now prove the following theorem, which is equivalent to Theorem \ref{th:r-singular}.
\begin{thm}\label{th:almost-simple}
Let $G$ be a finite almost simple group with simple socle $S$. Let $p$ be a prime and let $x\in G.$ If $[x,g]=1$ or $[x,g]$ is $p$-singular for every $g\in G$, then $x=1$.
\end{thm}

\begin{proof}
Let the pair $(G,x)$ be a counterexample to the theorem with $|G|$ minimal. Then $1\neq x\in G$ and $[x,g]=1$ or  $[x,g]$ is $p$-singular for every $g\in G.$ Suppose that $H$ is a finite almost simple proper subgroup of $G$ containing $x$. Clearly, $[x,h]=1$ or it is $p$-singular for every $h\in H.$  By the minimality of $|G|$, $x=1$, which is a contradiction. Thus $x$ does not lie in any almost simple proper subgroup of $G$. In particular,  $G=\la S,x\ra=S\la x\ra$ and hence $G'=S$. Let $K=x^G.$ Then 
\begin{equation}\label{eqn1} K^{-1}K=C_0\cup C_1\cup C_2\cup\dots\cup C_a
\end{equation} where $C_0=\{1\}$ and, for $1\leq i\leq a$, $C_i$ is the conjugacy class of a nontrivial $p$-singular element $y_i\in G$. Taking the class sums, we get 
\begin{equation}\label{eqn2}\widehat{K^{-1}}\widehat{K}=\sum_{i=0}^a n_i \widehat{C_i}\end{equation}
 where the structure constant $n_i=n(K^{-1},K,C_i)$ can be computed by \begin{equation}\label{eqn3}n(K^{-1},K,C_i)=\dfrac{|K|^2}{|G|}\sum_{\phi\in\Irr(G)}\dfrac{|\phi(x)|^2\overline{\phi(y_i)}}{\phi(1)}\end{equation} (see Exercise (3.9) in \cite{Isaacs-ct}). For $i=0,$ we have $n_0=|K|.$
Let $\chi\in\Irr(G)$. Applying the homomorphism $\omega_\chi$ to Equation \eqref{eqn2}, we get

\begin{equation}\label{eqn4}
\dfrac{|K|^2|\chi(x)|^2}{\chi(1)^2}=|K|+\sum_{i=1}^a n_i\dfrac{|C_i|\chi(y_i)}{\chi(1)}.
\end{equation}

(a) Assume first that $S$ is a sporadic simple group, the Tits group or an alternating group $\AAA_n$ with $5\leq n\leq 19$. In these cases, the character tables of the almost simple groups $G$ with socle $S$ are available in GAP \cite{GAP}. For each $1\neq x\in G,$ we can compute the structure constant $n(K^{-1},K,C)$ for the conjugacy classes $C=y^G$ for every $y\in G$, where $K=x^G,$ and one can check that there exist two nontrivial elements $z_1,z_2$ of coprime orders with nonzero structure constants. Hence $z_1^G$ and $z_2^G$ appear in the decomposition of the product $K^{-1}K$, which is a contradiction.

(b) Assume next that $S\cong \AAA_n$ with $n\ge 20.$ In this case, $G\in \{\AAA_n,\SSS_n\}.$ Let $\Omega=\{1,2,\ldots, n\}$. If $\sigma=(i_1,i_2,\dots,i_s)$ is an $s$-cycle, then we denote by $\ell(\sigma)$ the length of $\sigma$, which is $s$, and the support of $\sigma$ is $\supp(\sigma)=\{i_1,i_2,\ldots,i_s\}$. Let $x=\sigma_1\sigma_2\dots\sigma_k$, where $\sigma_1, \dots,\sigma_k$ are pairwise disjoint  cycles in $\SSS_n$, including cycles of length $1$. Write $\ell_i=\ell(\sigma_i)$, and assume without loss of generality that  $\ell_1\geq \ell_2\geq\dots\geq \ell_k.$ Note that $\sum_{i=1}^k \ell_i=n.$

Assume $x$ has a fixed point on $\Omega$. Then $\ell_k=1$. Without loss of generality, assume $\sigma_k=(n)$. Then $x=\sigma_1\sigma_2\dots\sigma_{k-1}\in G\cap \Sym(\Omega_1)$ with $\Omega_1=\Omega\setminus \{n\}$. Now $G\cap \Sym(\Omega_1)$ is a finite almost simple group with socle $\Alt(\Omega_1)\cong \AAA_{n-1}$ containing $x$, which is impossible.  Thus  $x$ has no fixed point and hence $\ell_k\ge 2.$

Assume  $G=\SSS_n$. Suppose that $k\ge 2.$ Then $\ell_k\leq n/2$. Let $x=uv$, where $v=\sigma_k$ and $u=\sigma_1\sigma_2\dots \sigma_{k-1}$. Then $1\neq u\in\Sym(\Omega_2)$ and $v\in\Sym(\Omega_1)$ with $\Omega_1=\supp(v)$ and $\Omega_2=\Omega\setminus \Omega_1.$  Here $n>m:=|\Omega_2|=n-\ell_k\ge n/2\ge 10$. For $g\in\Sym(\Omega_2)\leq G$, we have $[g,v]=1$, and thus $[x,g]=[uv,g]=[u,g]$ is either $1$ or $p$-singular. Thus $(\Sym(\Omega_2),u)$ satisfies the hypothesis of the theorem. By the minimality of $|G|$, we have $u=1$,  a contradiction. Suppose that $k=1$, so that $x$ is an $n$-cycle. Without loss of generality, assume $x=(1,2,\dots,n)$. Let $g_1=(1,3,2)$ and $g_2=(1,2)(3,4)$. Then both $g_1$ and $g_2$ lie in $\AAA_n\unlhd G$. Composing permutations from left to right, a direct computation gives $[x,g_1]=(1,3,4)$ and $[x,g_2]=(1,2,4,5,3)$. These nontrivial commutators have coprime orders $3$ and $5$, so they cannot both be $p$-singular, violating the hypothesis of the theorem.

Assume $G=\AAA_n.$ Observe that if $k=1$, then $x\in G$ is an $n$-cycle and $n$ must be odd. The argument above also shows that this is impossible. Thus, we assume $k\ge 2.$  Then $2\leq \ell_k\leq n/2$. Let $\Omega_1$ be the support of $\sigma_k$ and $\Omega_2=\Omega\setminus\Omega_1$. Then $m=|\Omega_2|=n-\ell_k\ge  n/2 \ge 10.$ Let $H=\Alt(\Omega_2)\cong \AAA_m$. Then $H$ is a finite nonabelian simple subgroup of $G$ and $H$ centralizes $v=\sigma_k$. Note that $x=uv=vu\in G.$

Assume $\ell_k$ is odd. Then $v$ is an even permutation. So $v\in\Alt(\Omega_1)$ and $u\in \Alt(\Omega_2)\leq G$. Let $g\in  H$. Then $[g,v]=1$ and thus $[x,g]=[u,g]$ is either $1$ or $p$-singular, and thus $u=1$ by the minimality of $|G|$, as $|H|<|G|$, a contradiction.

 Assume $\ell_k$ is even. Then $v$ is an odd permutation and so $u\in\Sym(\Omega_2)$ is odd as $x=uv$ is an even permutation.  Let $g\in\Sym(\Omega_2)$. If $g$ is an even permutation, then $g\in\Alt(\Omega_2)\leq G$; and if $g$ is an odd permutation, then $gv\in G$. As $[g,v]=1$, if $g$ is even, then $[u,g]=[x,g]$; and if $g$ is odd, then $[x,gv]=[uv,gv]=[u,g]$. In either case, $[u,g]$ is either $1$ or it is $p$-singular for all $g\in \Sym(\Omega_2)$. Hence the pair $(\Sym(\Omega_2),u)$ satisfies the hypothesis of the theorem with $|\Sym(\Omega_2)|<|G|$. Thus $u=1$, which is a contradiction.

(c) Assume that $S\not\in\{\PSp_4(2)', {}^2\textrm{F}_4(2)'\}$ is a finite simple group of Lie type in characteristic $r$, where $r$ is a prime.  Note that the cases $S\cong \PSp_4(2)'\cong \AAA_6$  and $S\cong {}^2\textrm{F}_4(2)'$ have been handled earlier. 

Assume first that $p=r$. Let $\theta\in\Irr(S)$ be the Steinberg character of $S$ of degree $|S|_p$. It is well known that $\theta$ extends to some $\chi\in\Irr(G)$ (see \cite{F1}) and that $\theta(y)=0$ whenever $y\in S$ is $p$-singular since $\theta$ has $p$-defect zero (\cite[Theorem 8.17]{Isaacs-ct}). For each $i$ with $1\leq i\leq a$, we know that $y_i=[x,g_i]$ for some $g_i\in G$. As $G/S$ is abelian,  $y_i\in G'=S$ for all $i\ge 1$. Moreover, $y_i$ is $p$-singular by the hypothesis. Therefore $\chi(y_i)=\theta(y_i)=0$ for all $i$ with $1\leq i\leq a$. From Equation \eqref{eqn4}, we obtain 
\begin{equation}\label{eqn5}
|\chi(x)|^2=\frac{\chi(1)^2}{|K|}=\frac{|S|_p^2}{|K|}.
\end{equation}
By \cite[Corollary 3.6]{Isaacs-ct}, $\chi(x)$ is an algebraic integer and so is $|\chi(x)|^2=\chi(x)\chi(x^{-1})$. By the previous equation, $|\chi(x)|^2$ is also a rational integer, and so it is an integer by \cite[Lemma 3.2]{Isaacs-ct}. It follows that $|K|$ divides $|S|_p^2$. In particular, $|K|$ is a power of $p$. By Lemma \ref{lem:pa}, $\la K\ra=\la x^G\ra$ is solvable. However, since $G$ is almost simple, $G$ has no nontrivial solvable normal subgroup. Thus $\la x^G\ra=1$, forcing $x=1$, a contradiction.

Assume that $p\neq r.$ If $p\nmid |S|$, then $S$ contains no $p$-singular elements;  so $[x,g]=1$ for all $g\in G$ and $x\in\ZB(G)=1$, a contradiction. Hence $p$ divides $|S|$. By Lemma \ref{lem:semisimple}, $S$ has an irreducible character $\theta$ of $p$-defect zero and $r\nmid \theta(1)$.  Let $\chi\in\Irr(G)$ be an irreducible character of $G$ lying above $\theta$. By Clifford's theorem (\cite[Theorem 6.2]{Isaacs-ct}), $\chi_S=e(\theta_1+\dots+\theta_t)$, where each $\theta_i$ is conjugate to $\theta=\theta_1$ and $t=|G:I_G(\theta)|$, where $I_G(\theta)$ is the inertia group of $\theta$ in $G$.  Since $\theta\in\Irr(S)$ has $p$-defect zero, each $\theta_i$ also has $p$-defect zero and thus $\theta_i(y)=0$ for every $p$-singular element $y\in S.$ As in the previous case, each $y_j$ with $1\leq j\leq a$ lies in $S$ and is $p$-singular. Thus $\theta_i(y_j)=0$ for every $i$ and $j$ with $1\leq i\leq t$ and $1\leq j\leq a.$ It follows that $\chi(y_j)=e\sum_{i=1}^t\theta_i(y_j)=0$ for all $j$ with $1\leq j\leq a$. Equation \eqref{eqn4} yields
\begin{equation}\label{eqn6}
|\chi(x)|^2=\frac{\chi(1)^2}{|K|}.
\end{equation}
We claim that $\chi$ is an extension of $\theta$.  If $t>1$, then $\chi=\varphi^G$ for some $\varphi\in\Irr(I_G(\theta))$ lying over $\theta$. In this case, $S\unlhd I_G(\theta)\unlhd G$, as $G/S$ is abelian, and $x\not\in I_G(\theta)$ since $G=S\la x\ra$ and $I_G(\theta)\neq G$. This implies that $\chi(x)=\varphi^G(x)=0$. However, this is impossible by Equation \eqref{eqn6}. Therefore $t=1$ and so $\theta$ is $G$-invariant. Since $G/S$ is cyclic, by \cite[Corollary 11.22]{Isaacs-ct} $\theta$ extends to $G$ and thus $\chi$ is an extension of $\theta$. Hence $\chi(1)=\theta(1)$ is not divisible by $r$. Recall that $r$ is the characteristic of $S$.

As in the previous case, $|\chi(x)|^2$ is both an algebraic integer and a rational integer, so it is an integer, and thus $|K|$ divides $\chi(1)^2$. Hence $r$ does not divide $|K|=|G:\CB_G(x)|$. Thus $\CB_G(x)$ contains a Sylow $r$-subgroup $R$ of $G$. In other words, $x\in \CB_G(R)$. By Lemma \ref{lem:Sylow-centralizers}, $x\in R$. In particular, $x$ is a nontrivial $r$-element.    
 
 Suppose that $x^G\cap \CB_G(x)=\{x\}$. Since $x$ is an $r$-element, by Glauberman's $\ZB_p^*$-theorem (\cite[Theorem 5.1]{GTT}), $x\in \ZB_r^*(G)$. However, as $G$ is an almost simple group with simple socle $S$, a finite simple group of Lie type in characteristic $r$, $\OB_{r'}(G)=1=\ZB(G)$. Hence $x\in \ZB_r^*(G)=1$, a contradiction.  Thus we can find $g\in G$ such that $x^g\neq x$ and $[x,x^g]=1$. Hence $[x,g]=x^{-1}x^g$ is a nontrivial $r$-element. On the other hand, $[x,g]$ is $p$-singular. Thus $p=r$, contradicting our assumption that $p\neq r.$  The proof is now complete.
\end{proof}

 We are ready to prove Theorem \ref{th:r-elements-commutators}. 
 
 \begin{proof}[\textbf{Proof of Theorem \ref{th:r-elements-commutators}}]
 Suppose first that $x$ is central modulo $\OB_p(G)$. Then for every $g\in G$, we have $[x,g]\in \OB_p(G)$, and the result follows. 
 We now prove the converse. Assume that $[x,g]$ is a $p$-element for every $g\in G$. We will show, by induction on $|G|$, that $x$ is central modulo $\OB_p(G)$; that is, $[x,g]\in \OB_p(G)$ for all $g\in G.$

 Let $\overline{G}=G/\OB_p(G)$ and adopt the `bar' notation. Note that $\OB_p(\overline{G})=1$. For every $\overline{g}\in\overline{G}$, we have $[\overline{x},\overline{g}]=[x,g]\OB_p(G)$, which  is a $p$-element in $\overline{G}$. Suppose $\OB_p(G)>1$. Then $|\overline{G}|=|G/\OB_p(G)|<|G|$, so by the induction hypothesis, $\overline{x}\in \ZB(\overline{G})$. Hence, ${x}$ is central in ${G}$ modulo $\OB_p(G)$ as desired.  Now assume $\OB_p(G)=1$. We will show that $x\in \ZB(G)$.
 
Assume $M=\la x^G\ra<G$. Then $M\unlhd G$ and since $\OB_p(G)=1$, we have $\OB_p(M)\leq \OB_p(G)=1$, so $\OB_p(M)=1$. For every $g\in M$, the commutator  $[x,g]$ is a $p$-element by hypothesis. Thus, by the induction hypothesis,  $x\in \ZB(M)$. Since $\OB_p(M)=1$,  $\ZB(M)$ is a  $p'$-subgroup of $M$, and thus $x\in \OB_{p'}(M)\leq \OB_{p'}(G)$. Therefore, for all $g\in G$,  $[x,g]\in \OB_{p'}(G)$. But since $[x,g]$ is a $p$-element by assumption, it must be trivial. Thus $[x,g]=1$ for all $g\in G$, and hence $x\in \ZB(G)$.

 Thus we may assume that $\la x^G\ra=G$. Let $N$ be a minimal normal subgroup of $G$.  We consider the following cases.
 
 \smallskip
 \textbf{Case $1$}. Assume $N$ is abelian. Then $N$ is an elementary abelian $r$-group for some prime $r\neq p$. Since $[x,n]\in N$ is an $r$-element for all $n\in N$, it follows that $[x,N]=1$, so $x\in \CB_G(N).$ As $\CB_G(N)\unlhd G$ and $\la x^G\ra=G$, $\CB_G(N)=G$, so $N\leq \ZB(G).$ Now $xN\in G/N$ satisfies the condition that $[xN,gN]$ is a $p$-element for all $gN\in G/N.$ We claim that $\OB_p(G/N)=1$. Once this is established, we may apply the induction hypothesis to conclude that $xN\in \ZB(G/N)$, which implies $[x,g]\in N$ for all $g\in G$. Since $[x,g]$ is a $p$-element for every $g\in G$ by hypothesis, it follows that $[x,g]=1$ for all $g\in G$, so  $x\in \ZB(G)$ as wanted.  To prove the claim, let $U/N=\OB_p(G/N)$. Since $N\leq \ZB(G)$, $N$ is a central Sylow $r$-subgroup of $U$. It follows that $U$ has a normal Sylow $p$-subgroup $P$. Since $U\unlhd G$, we have $P\unlhd G$  as well. As $\OB_p(G)=1$, $P=1$. Hence $\OB_p(G/N)=1$, completing the proof of the claim.

  Thus, we may assume that $\RB(G)=1$.

\smallskip
 \textbf{Case $2$}.  Assume $N$ is nonabelian. Then $N\cong S^k$, where $S$ is a finite nonabelian simple group and $k\ge 1$. 
Suppose that $N$ is a $p'$-group. Then $[x,N]=1$, so  $N\leq \ZB(G)$, a contradiction.  Hence $p$ divides $|S|$. Let $H=\la x\ra N$, and suppose $H<G$. By induction, $x$ is central in $H$ modulo $\OB_p(H)$, so $\la x\ra \OB_p(H)\unlhd H$, and hence $x\in \RB(H)$. Since $\RB(H)\cap N=1$, we have $[\RB(H),N]=1$,  and so $x\in \CB_G(N)\unlhd G$, contradicting the fact that $\la x^G\ra=G$ and $N$ is nonabelian. Thus $G=\la x\ra N.$ 

Let $C=\CB_G(N). $ Then $C\unlhd G.$ Assume $C>1$. Then  $C\cap N=1$ and $[C,N]=1$. As $G=\la x\ra N$,  $C\cong CN/N\leq G/N$ and so $C$ is a cyclic normal subgroup of $G$. In particular, $C\leq \RB(G)=1$, a contradiction. Therefore, $N$ is the unique minimal normal subgroup of $G$. Write $N=S_1\times S_2\times \cdots\times S_k$, where  $S_i\cong S$ for $1\leq i\leq k$. Note that $\la x\ra$ acts transitively on the set $\{S_i\}_{i=1}^k$.
Let $L=S_1$. Assume that $k\ge 2.$ Then $L\neq L^x$, and $L$ is a component of $G$. By Burnside's $p^aq^b$-theorem (\cite[Theorem 3.10]{Isaacs-ct}), $|L|$ has a prime divisor $s\neq p$. By Lemma \ref{lem:component-moved}, there exists an $s$-element $a\in L$ such that $[x,a]$ is a nontrivial $s$-element. However, $[x,a]$ is a $p$-element by hypothesis, a contradiction. Therefore $x$ normalizes $L$, which forces $k=1$. Hence $G$ is a finite almost simple group with simple socle $S$. 
By Theorem \ref{th:almost-simple}, we conclude that $x=1\in \ZB(G)$, completing the proof of the theorem.
 \end{proof}

\section{Products of conjugacy classes}\label{sec4}
For a finite group $G$, we denote the last term of the derived series of $G$ by $G^{(\infty)}$.
We now prove Theorems \ref{th:3-classes} and \ref{th:two-primes-divisors}.

\begin{proof}[\textbf{Proof of Theorem \ref{th:3-classes}}]
Let $G$ be a counterexample to the theorem with $|G|$ minimal.  Then there exists a conjugacy class $K$ of $G$ such that $K^{-1}K=1\cup D\cup D^{-1}$ for some conjugacy class $D$ of $G$, but $H=\la K\ra$ is not solvable. Let $x\in K$ and $y\in D$. Note that  $D=y^G$ and $D^{-1}=(y^{-1})^G.$

Let $N$ be a nontrivial normal subgroup of $G$. The conjugacy class $(xN)^{G/N}$ in $G/N$ also satisfies the hypothesis of the theorem. By the minimality of $|G|$, we conclude that $HN/N\cong H/(H\cap N)$ is solvable. In particular, if $N$ is solvable, then so is $H.$ Therefore, we may assume $\RB(G)=1$.

Suppose $G$ has two distinct minimal normal subgroups $N_1$ and $N_2$. Then $N_1\cap N_2=1$ and the previous argument implies that  $H^{(\infty)}\subseteq N_i$ for each $i=1,2$. Hence $H^{(\infty)}\subseteq N_1\cap N_2=1$, so $H$ is solvable, a contradiction. Thus, $G$ has a unique minimal normal subgroup $N\cong S^k$, where $S$ is a finite nonabelian simple group and $k\ge 1$. It follows that $\CB_G(N)=1$, and so $\FB^*(G)=\EB(G)=N$. 

Recall that $K^{-1}K=1\cup D\cup D^{-1}$ and let $m=o(y)$. For every $g\in G$, $[x,g]=x^{-1}x^g\in K^{-1}K$, so either  $[x,g]=1$ or $[x,g]$ is conjugate in $G$ to $y$ or $y^{-1}$. In particular,  $[x,g]=1$ or $o([x,g])=m$ for all $g\in G.$  If $m$ is a power of a prime $p$ (this includes $m=1$, in which case $x\in\ZB(G)$), then Theorem \ref{th:r-elements-commutators} yields that $\la x\ra \OB_p(G)\unlhd G$, and hence $H$ is solvable, a contradiction. So we assume  that $m>1$ is not a prime power.

Let $L$ be a component of $G$. Then $L\leq N$ and $L\cong S.$  Assume that $L^x\neq L$ and let $p$ be a prime divisor of $|L|$. By Lemma \ref{lem:component-moved}, there exists a $p$-element $a\in L$ such that $[x,a]$ is a nontrivial $p$-element. However, $[x,a]\neq 1$ has order $m$, which is not a prime power, a contradiction. Hence $L^x=L$ for every component $L$ of $G$. If $x$ centralizes all components of $G$, then $x\in \CB_G(\FB^*(G))=\CB_G(N)=1$, a contradiction. Therefore, $x$ normalizes but does not centralize a component $L$ of $G$.

Let $M=\la x\ra L$ and let $A=\CB_M(L)\unlhd M$.  Then  $L\cap A=1$ and $\overline{M}=M/A$ is  almost simple  with socle $\overline{L}=LA/A\cong L$. Let $g\in M$ be such that $\overline{g}\in \overline{M}\setminus \CB_{\overline{M}}(\overline{x})$. Then  $z=[x,g]=x^{-1}x^g\neq 1$. Since $1\neq z\in K^{-1}K=1\cup D\cup D^{-1}$, we have $z\in D\cup D^{-1}$. Thus $o(z)=m>1$, and $t:=o(\overline{z})$ divides $m$. Since $M=\la x\ra L$, we have $M'=L$, whence $z\in L$. On the other hand, $z^t\in A$. So $z^t\in L\cap A=1$, which implies $t=m$. We have shown that for all $\overline{g}\in\overline{M}$, either $[\overline{x},\overline{g}]=\overline{1}$  or $[\overline{x},\overline{g}]$ has order $m$. Since $m>1$ is divisible by some prime, Theorem \ref{th:almost-simple} implies $\overline{x}=\overline{1}$. Therefore, $x\in A$ centralizes $L$, contradicting our earlier conclusion. This contradiction completes the proof.
\end{proof}

\begin{proof}[\textbf{Proof of Theorem \ref{th:two-primes-divisors}}]

Let $G$ be a counterexample to the theorem with $|G|$ minimal. Then $x$ is a $p$-element for some prime $p$ and there exist two distinct primes $r$ and $s$ such that  for every $g\in G$, either $[x,g]=1$ or $rs$ divides $o([x,g])$; yet $x\not\in \ZB(G)$. 
By the minimality of $|G|$, if $H$ is a proper subgroup of $G$ containing $x$, then $x\in \ZB(H)$.

Let $M=\la x^G\ra$, and suppose $M<G.$ Then by the previous paragraph,  we have $x\in \ZB(M)\subseteq \FB(G)$. Since $x$ is a $p$-element, $x\in \OB_p(G)$. However, since $x\not\in \ZB(G)$, there exists  some $g\in G\setminus \CB_G(x)$, so  $1\neq [x,g]\in \OB_p(G)$. This contradicts the hypothesis, since $o([x,g])$ is a power of $p$. Hence, we must have $G=\la x^G\ra.$
 
 Assume that $G$ is nonsolvable. Then $x\in \ZB(H)$ for every solvable subgroup  $H$ of $G$ containing $x$. By Lemma \ref{lem:central}, $x\in \ZB(G)$, contradicting our assumption. Therefore, $G$ is solvable.

Now consider the normal subgroups $\OB_r(G)$ and $\OB_{r'}(G)$. Suppose $g\in \OB_r(G)\setminus \CB_{G}(x)$. Then $[x,g]\neq 1$ is an $r$-element, hence $rs\nmid o([x,g])$, a contradiction. The same argument applies to $\OB_{r'}(G)$. Therefore,  $x$ centralizes both $\OB_r(G)$ and $\OB_{r'}(G)$. Since $x^G$ generates $G$, both $\OB_r(G)$ and $\OB_{r'}(G)$ are central in $G$. In particular, $\FB(G)\leq \ZB(G)$. By Fitting's lemma (\cite[31.10]{Aschbacher}), $\CB_G(\FB(G))\leq  \FB(G)$, so it follows that $G=\FB(G)\leq\ZB(G)$. Hence $G$ is abelian and $x\in \ZB(G)$, a contradiction. This completes the proof.
\end{proof}

As a final remark, we note that  an unpublished result due to Guralnick and Tracey \cite{GTr}, in which the authors classify all pairs $(G,x)$ such that  $G$ is a finite group and $x\in G$ lies in a unique maximal subgroup of $G$, can be used to prove the following  variation of Theorem \ref{th:two-primes-divisors} :  If $x$ is an element of  a finite group $G$ and there exist  two distinct primes  $r$ and $s$ such that for every $g\in G$,  either $[x,g]=1$ or the order of $[x,g]$  is divisible by $rs$, then $x\in \FB(G)$.
 We point out that the  stronger conclusion $x\in \ZB(G)$ need not hold if $x$ does not have prime power order. For instance, let  $G=\texttt{SmallGroup}(24,4)$ in MAGMA \cite{magma}, which is isomorphic to $\textrm{C}_3\rtimes \textrm{Q}_8$, where $\ZB(G)\cong \textrm{C}_2$ and $\FB(G)=\la x\ra \cong \textrm{C}_{12}$. One can check that $[x,g]=1$ or $o([x,g])=6$ for all $g\in G$ yet $x\not\in \ZB(G)$.
As a consequence, we obtain the following result:  If $x\in G$ and there exists an integer $m\ge 1$ such that  either $[x,g]=1$ or the order of $[x,g]$ is $m$ for every $g\in G$, then $\la x^G\ra$ is solvable. 

\subsection*{Acknowledgment} The author thanks Gunter Malle for carefully reading an earlier version of the paper and for his helpful comments. He is also grateful to Bob Guralnick and Mandi Schaeffer Fry for discussions during the preparation of this work.


\begin{thebibliography}{999}

\bibitem{AL} J.~L. Alperin and R.~N. Lyons, On conjugacy classes of $p$-elements, J. Algebra {\bf 19} (1971), 536--537.

\bibitem{AH} Z. Arad, M. Herzog, eds. Products of conjugacy classes in groups, Lecture Notes in Mathematics,  \textbf{1112}. Springer-Verlag,  Berlin, 1985.


\bibitem{Ar} O. D. Artemovich, {Isolated elements of prime order in finite groups}, {Ukrainian Math. J.} {\bf 40} (1988), no.~3, 343--345 (1989); translated from Ukrain. Mat. Zh. \textbf{40} (1988), no. 3, 397--400, 408. 


\bibitem{Aschbacher} M.~G. Aschbacher, {Finite group theory}, second edition, 
Cambridge Studies in Advanced Mathematics, 10, Cambridge Univ. Press, Cambridge, 2000.

\bibitem{Baer} R. Baer, Engelsche Elemente Noetherscher Gruppen, Math. Ann. {\bf 133} (1957), 256--270.

\bibitem{BFM3} A. Beltr\'an, M.~J. Felipe and C. Melchor, Multiplying a conjugacy class by its inverse in a finite group, {Israel J. Math.} {\bf 227} (2018), no.~2, 811--825.

\bibitem{BFM4} A. Beltr\'an, M.~J. Felipe and C. Melchor, Squares of real conjugacy classes in finite groups, {Ann. Mat. Pura Appl.} (4) {\bf 197} (2018), no.~2, 317--328.


\bibitem{BFM2} A. Beltr\'an, M.~J. Felipe and C. Melchor, Some problems about products of conjugacy classes in finite groups, Int. J. Group Theory {\bf 9} (2020), no.~1, 59--68.

\bibitem{magma}  W. Bosma, J. Cannon and C. Playoust, {The {\textsc{Magma}} algebra system I: The user language}, J. Symb. Comput. \textbf{24} (1997), 235--265.

\bibitem{Carter} R.~W. Carter, { Finite groups of Lie type}, Pure and Applied Mathematics (New York) A Wiley-Interscience Publication,  Wiley, New York, 1985.

\bibitem{ATLAS} J. H. Conway, R. T. Curtis, S. P. Norton, R. A. Parker and R. A. Wilson, {Atlas of  {F}inite {G}roups}, Oxford University Press, 1985.

\bibitem{Dornhoff} L.~L. Dornhoff, {Group representation theory. Part A}, Pure and Applied Mathematics, 7, Marcel Dekker, Inc., New York, 1971. 

\bibitem{Enomoto}
H. Enomoto, The characters of the finite symplectic group ${\rm Sp}(4,\,q)$, $q=2\sp{f}$, Osaka Math. J. {\bf 9} (1972), 75--94.


\bibitem{F1} W. Feit, Extending Steinberg characters, in { Linear algebraic groups and their representations (Los Angeles, CA, 1992)}, 1--9, Contemp. Math., 153, Amer. Math. Soc., Providence, RI, 1993.

\bibitem{FA} E. Fisman\ and\ Z. Arad, A proof of Szep's conjecture on nonsimplicity of certain finite groups, {J. Algebra} {\bf 108} (1987), no.~2, 340--354. 

\bibitem{Glauberman} G. Glauberman, Central elements in core-free groups, {J. Algebra} \textbf{4} (1966), 403--420.

\bibitem{GLS3} D. Gorenstein, R. Lyons\ and\ R. Solomon, {The classification of the finite simple groups. Number 3. Part I}, Mathematical Surveys and Monographs, 40.3, American Mathematical Society, Providence, RI, 1998. 

\bibitem{GAP} The GAP Group, {GAP -- Groups, Algorithms, and Programming, Version 4.11.0}, 2020, \url{(http://www.gap-system.org)}.


\bibitem{GM1} R. Guralnick\ and\ G. Malle, Variations on the Baer--Suzuki theorem, {Math. Z.} \textbf{ 279} (2015), no.~3-4, 981--1006. 


\bibitem{GMT} R. M. Guralnick, G. Malle\ and\ P. H. Tiep, Products of conjugacy classes in finite and algebraic simple groups, {Adv. Math.} {\bf 234} (2013), 618--652.

\bibitem{GN} R.~M. Guralnick and G. Navarro, Squaring a conjugacy class and cosets of normal subgroups, {Proc. Amer. Math. Soc.} {\bf 144} (2016), no.~5, 1939--1945.


\bibitem{GR} R. M. Guralnick\ and\ G. R. Robinson, On extensions of the Baer--Suzuki theorem, {Israel J. Math.} \textbf{ 82} (1993), no.~1-3, 281--297. 


\bibitem{GTT}  R.~M. Guralnick, H.~P. Tong-Viet\ and\ G.~M. Tracey, Weakly subnormal subgroups and variations of the Baer--Suzuki theorem, J. Lond. Math. Soc. (2) {\bf 111} (2025), no.~1, Paper No. e70057, 29 pp.

 \bibitem{GTr}  R. Guralnick and G. Tracey,   Elements contained in a unique maximal subgroup, preprint available at \url{https://sites.google.com/view/gareth-tracey/publications-and-preprints?authuser=0}.

\bibitem{Isaacs-ct} I.~M. Isaacs, {Character theory of finite groups},  AMS Chelsea Publ., Providence, RI, 2006.


\bibitem{Isaacs-fg} I. M. Isaacs, {Finite group theory}, Graduate Studies in Mathematics, 92, American Mathematical Society, Providence, RI, 2008.

\bibitem{Kazarin} L.~S. Kazarin,  Burnside's $p^\alpha$-Lemma, {Math. Notes} {\bf 48} (1990), no.~1-2, 749--751 (1991); translated from {Mat. Zametki} {\bf 48} (1990), no. 2, 45--48, 158. 


\bibitem{LT2}  M.~J. Larsen\ and\ P.~H. Tiep, Character estimates for finite classical groups and the asymptotic Thompson Conjecture, \url{https://arxiv.org/abs/2403.09047}.

\bibitem{LT1} M.~J. Larsen\ and\ P.~H. Tiep, Uniform character bounds for finite classical groups, Ann. of Math. (2) {\bf 200} (2024), no.~1, 1--70. 

\bibitem{LWZ} J. Liao, L. Wang\ and\ J. Zhang, Thompson's conjecture for finite simple groups of Lie type over small fields, preprint (2026), \url{https://arxiv.org/abs/2608.17891}.

\bibitem{Lusztig} G. Lusztig, On the representations of reductive groups with disconnected centre, Ast\'erisque No. \textbf{168} (1988), 10, 157--166.

\bibitem{Malle} G. Malle, Zeros of Brauer characters and the defect zero graph, J. Group Theory {\bf 13} (2010), no.~2, 171--187.

\bibitem{Michler} G.~O. Michler, A finite simple group of Lie type has $p$-blocks with different defects, $p\not=2$, J. Algebra {\bf 104} (1986), no.~2, 220--230.

\bibitem{MTong} J. Moori and H.~P. Tong-Viet, Products of conjugacy classes in simple groups, {Quaest. Math.} {\bf 34} (2011), no.~4, 433--439. 

\bibitem{Shult} E. Shult, Some analogues of Glauberman's $Z\sp{\ast} $-theorem, {Proc. Amer. Math. Soc.} {\bf 17} (1966), 1186--1190.

\bibitem{SF} W.~A. Simpson and J.~S. Frame, The character tables for ${\rm SL}(3,\,q)$, ${\rm SU}(3,\,q\sp{2})$, ${\rm PSL}(3,\,q)$, ${\rm PSU}(3,\,q\sp{2})$, Canadian J. Math. {\bf 25} (1973), 486--494.

\bibitem{Srinivasan} B. Srinivasan, The characters of the finite symplectic group ${\rm Sp}(4,\,q)$, Trans. Amer. Math. Soc. {\bf 131} (1968), 488--525.


\bibitem{Suzuki} M. Suzuki, Finite groups in which the centralizer of any element of order $2$ is $2$-closed, Ann. of Math. (2) {\bf 82} (1965), 191--212.


\bibitem{Wielandt} H. Wielandt, Kriterien f\"{u}r Subnormalit\"{a}t in endlichen Gruppen, {Math. Z.} \textbf{ 138} (1974), 199--203.

\bibitem{Willems} W. Willems, Blocks of defect zero in finite simple groups of Lie type, J. Algebra {\bf 113} (1988), no.~2, 511--522.


\end{thebibliography}
\end{document}